\documentclass[12pt, twoside]{article}
\usepackage{amsmath,amssymb,amsthm}
\usepackage{bm}
\usepackage[dvipdfm]{pict2e}

\usepackage[dvipdfmx]{graphicx,color}
\makeatletter

\@addtoreset{equation}{section}
\makeatother
\theoremstyle{plain} 
\newtheorem{theorem}{Theorem}
\newtheorem*{theorem*}{Theorem}
\newtheorem{prop}[theorem]{Proposition}
\newtheorem*{prop*}{Proposition}
\newtheorem{lemma}[theorem]{Lemma}
\newtheorem*{lemma*}{Lemma}
\newtheorem{cor}[theorem]{Corollary}
\newtheorem*{cor*}{Corollary}

\newtheorem*{example*}{Example}

\newtheorem*{axiom*}{Axiom}

\newtheorem*{problem*}{Problem}

\newtheorem*{summary*}{Summary}

\newtheorem*{guide*}{Guide}
\theoremstyle{definition} 
\newtheorem{definition}[theorem]{Definition}
\newtheorem*{definition*}{Definition}
\theoremstyle{definition} 
\newtheorem{remark}[theorem]{Remark}
\newtheorem*{remark*}{Remark}
\numberwithin{theorem}{section}
\numberwithin{equation}{section}
\numberwithin{figure}{section}
\numberwithin{table}{section}

\makeatletter
\renewenvironment{proof}[1][\proofname]{\par
  \normalfont
  \topsep6\p@\@plus6\p@ \trivlist
  \item[\hskip\labelsep{\bfseries #1}\@addpunct{\bfseries.}]\ignorespaces
}{%
  \endtrivlist
}
\renewcommand{\proofname}{Proof}
\makeatother
\makeatletter
\def\BOXSYMBOL{\RIfM@\bgroup\else$\bgroup\aftergroup$\fi
  \vcenter{\hrule\hbox{\vrule height.85em\kern.6em\vrule}\hrule}\egroup}
\makeatother
\newcommand{\BOX}{%
  \ifmmode\else\leavevmode\unskip\penalty9999\hbox{}\nobreak\hfill\fi
  \quad\hbox{\BOXSYMBOL}}
\renewcommand\qed{\BOX}

\begin{document}
\title{\textbf{Elliptic Ding-Iohara Algebra \\ and Commutative Families of \\ the Elliptic Macdonald Operator}}
\author{\textbf{Yosuke Saito} \\ Mathematical Institute, Tohoku University, \\ Sendai, Japan}
\maketitle

\begin{abstract}
The elliptic Ding-Iohara algebra is an elliptic quantum group obtained from the free field realization of the elliptic Macdonald operator. In this article, we show the construction of two families of commuting operators which contain the elliptic Macdonald operator by using the elliptic Ding-Iohara algebra and the elliptic Feigin-Odesskii algebra.
\end{abstract}
\tableofcontents
\pagestyle{headings}

\vskip 0.5cm 
\textbf{Notations.} In this paper, we use the following symbols.
\begin{align*}
&\mathbf{Z} : \text{The set of integers}, \quad \mathbf{Z}_{\geq{0}}:=\{0,1,2,\cdots\}, \quad \mathbf{Z}_{>0}:=\{1, \, 2, \, \cdots\}, \\
&\mathbf{Q} : \text{The set of rational numbers}, \quad \mathbf{Q}(q,t) : \text{The field of rational functions of $q$, $t$ over $\mathbf{Q}$}, \\
&\mathbf{C} : \text{The set of complex numbers}, \quad \mathbf{C}^{\times}:=\mathbf{C}\setminus\{0\}, \\
&\mathbf{C}[[z,z^{-1}]] : \text{The set of formal power series of $z$, $z^{-1}$ over $\mathbf{C}$}.
\end{align*} 

If a sequence $\lambda=(\lambda_{1},\cdots,\lambda_{N})\in{(\mathbf{Z}_{\geq{0}})^{N}}$ $(N\in\mathbf{Z}_{>0})$ satisfies the condition $\lambda_{i}\geq{\lambda_{i+1}}$ $(1\leq{i}\leq{N})$, $\lambda$ is called a partition. We denote the set of partitions by $\mathcal{P}$. For a partition $\lambda$, $\ell(\lambda):=\sharp\{i : \lambda_{i}\neq 0\}$ denotes the length of $\lambda$ and $|\lambda|:=\sum_{i=1}^{\ell(\lambda)}\lambda_{i}$ denotes the size of $\lambda$. 

Let $q$, $p\in\mathbf{C}$ be complex parameters satisfying $|q|<1$, $|p|<1$. We define the $q$-infinite product as $(x;q)_{\infty}:=\prod_{n\geq{0}}(1-xq^{n})$ and the theta function as 
\begin{align*}
\Theta_{p}(x):=(p;p)_{\infty}(x;p)_{\infty}(px^{-1};p)_{\infty}.
\end{align*} 
We set the double infinite product as $(x;q,p)_{\infty}:=\prod_{m,n\geq{0}}(1-xq^{m}p^{n})$ and the elliptic gamma function as 
\begin{align*}
\Gamma_{q,p}(x):=\frac{(qpx^{-1};q,p)_{\infty}}{(x;q,p)_{\infty}}.
\end{align*}
For the theta function $\Theta_{p}(x)$ and the elliptic gamma function $\Gamma_{q,p}(x)$, we have the following relations.
\begin{align*}
&\Theta_{p}(x)=-x\Theta_{p}(x^{-1}), \quad \Theta_{p}(px)=\Theta_{p}(x^{-1})=-x^{-1}\Theta_{p}(x), \\
&\Gamma_{q,p}(qx)=\frac{\Theta_{p}(x)}{(p;p)_{\infty}}\Gamma_{q,p}(x), \quad \Gamma_{q,p}(px)=\frac{\Theta_{q}(x)}{(q;q)_{\infty}}\Gamma_{q,p}(x).
\end{align*}

\section{Introduction}
Main topics of this article are about the elliptic Ding-Iohara algebra and the construction of commutative families of the elliptic Macdonald operator. Before starting discussions about them, we give backgrounds of this article.

The Ding-Iohara algebra was introduced by Ding and Iohara as a generalization of the quantum affine algebra $U_{q}(\widehat{sl_{2}})$ [6]. Defining relations of the algebra involve the structure function $g(x)$ which satisfies $g(x^{-1})=g(x)^{-1}$. Hence we can understand that for each function which satisfies $g(x^{-1})=g(x)^{-1}$, there is a Ding-Iohara algebra whose structure function is given by $g(x)$. 

In 2009, Feigin, Hashizume, Hoshino, Shiraishi, and Yanagida found that from the free field realization of the Macdonald operator, a kind of Ding-Iohara algebra was obtained [8]. In the following, we use the word "Ding-Iohara algebra" as the quantum group arising from the free field realization of the Macdonald operator. We can notice that Miki also reached the Ding-Iohara algebra in a different way [7]. There are some applications of the Ding-Iohara algebra such as the AGT conjecture [10][11][12], the refined topological vertex [13].

On the other hand, there exists an elliptic analog of the Macdonald operator [1][4][9]. In [8], Feigin, Hashizume, Hoshino, Shiraishi, and Yanagida made an attempt to construct the free field realization of the elliptic Macdonald operator by the idea of the quasi-Hopf twist. They obtained an elliptic analog of the Ding-Iohara algebra, however some problems remained open. Thus the author have started to search other ways of the free field realization of the elliptic Macdonald operator.

Then the author paid attention to the fact that the free field realization of the Macdonald operator is based on the form of the kernel function of the operator. Furthermore Komori, Noumi, and Shiraishi obtained the kernel function of the elliptic Macdonald operator [9]. Hence it would be a natural expect that the elliptic kernel function has important informations for the free field realization of the elliptic Macdonald operator. Consequently, starting from the elliptic kernel function, the free field realization of the elliptic Macdonald operator was obtained, and another elliptic analog of the Ding-Iohara algebra arose. We call the elliptic algebra the elliptic Ding-Iohara algebra.

\begin{align*}
\begin{matrix}
\text{Elliptic Macdonald operator} & \xrightarrow[]{\textbf{free field realization !}}& \textbf{Elliptic Ding-Iohara algebra} \\
\bigg{\uparrow}\text{\footnotesize{elliptic deformation}} & {} & \bigg{\uparrow}\text{\footnotesize{\textbf{elliptic deformation !}}} \\
\text{Macdonald operator} & \xrightarrow[]{\text{\quad free field realization \quad}}& \text{Ding-Iohara algebra}
\end{matrix}
\end{align*}

In the paper [8], Feigin, Hashizume, Hoshino, Shiraishi and Yanagida also constructed two families of commuting operators which contain the Macdonald operator (commutative families of the Macdonald operator). For the construction of the commutative families of the Macdonald operator, they used the Ding-Iohara algebra and the trigonometric Feigin-Odesskii algebra. Moreover the elliptic Feigin-Odesskii algebra which is an elliptic analog of the trigonometric Feigin-Odesskii algebra was defined. Hence we can consider to combine the elliptic Ding-Iohara algebra and the elliptic Feigin-Odesskii algebra. Actually we can construct two commutative families of the elliptic Macdonald operator by using the elliptic Ding-Iohara algebra and the elliptic Feigin-Odesskii algebra.

\medskip
\textbf{Organization of this paper.}

Organization of this paper is as follows. Section 2 is overview of the trigonometric case [8]. In section 3 we give the free field realization of the elliptic Macdonald operator and introduce the elliptic Ding-Iohara algebra [14]. Construction of the commutative families of the elliptic Macdonald operator by using the elliptic Ding-Iohara algebra and the elliptic Feigin-Odesskii algebra is shown [15].

\section{Trigonometric theory side}
In this section, we review some results due to Feigin, Hashizume, Hoshino, Shiraishi, and Yanagida [8] : the free field realization of the Macdonald operator, the Ding-Iohara algebra, and the construction of the commutative families of the Macdonald operator.

\subsection{Free field realization of the Macdonald operator}
In the following, let $q$, $t\in{\mathbf{C}}$ be parameters and we assume $|q|<1$. First we define the algebra $\mathcal{B}$ of boson to be generated by $\{a_{n}\}_{n\in{\mathbf{Z}\setminus\{0\}}}$ and the relation :
\begin{align}
[a_{m},a_{n}]=m\frac{1-q^{|m|}}{1-t^{|m|}}\delta_{m+n,0}.
\end{align}
We set the normal ordering $\bm{:} \bullet \bm{:}$ as 
\begin{align*}
\bm{:}a_{m}a_{n}\bm{:}=
\begin{cases}
a_{m}a_{n} \quad (m<n), \\
a_{n}a_{m} \quad (m\geq{n}).
\end{cases}
\end{align*}
Let $|0 \rangle$ be the vacuum vector which satisfies $a_{n}|0 \rangle=0$ $(n>0)$. For a partition $\lambda$, we set $a_{-\lambda}:=a_{-\lambda_{1}}\cdots a_{-\lambda_{\ell(\lambda)}}$ and define the boson Fock space $\mathcal{F}$ as a left $\mathcal{B}$ module :
\begin{align*}
\mathcal{F}:=\text{span}\{a_{-\lambda}|0 \rangle : \lambda\in{\mathcal{P}}\}.
\end{align*}

Set the $q$-shift operator as $T_{q,x}f(x):=f(qx)$. We define the Macdonald operator $H_{N}(q,t)$ $(N\in\mathbf{Z}_{>0})$ [1][2][5] as
\begin{align}
H_{N}(q,t):=\sum_{i=1}^{N}\prod_{j\neq{i}}\frac{tx_{i}-x_{j}}{x_{i}-x_{j}}T_{q,x_{i}}.
\end{align}

In the following $[f(z)]_{1}:=\displaystyle \oint\frac{dz}{2\pi iz}f(z)$ denotes the constant term of $f(z)$ in $z$.

\begin{prop}[\textbf{Free field realization of the Macdonald operator [5][8]}] 
Set operators $\eta(z)$, $\xi(z)$, $\phi(z) : \mathcal{F} \to \mathcal{F}\otimes\mathbf{C}[[z,z^{-1}]]$ as follows ($\gamma:=(qt^{-1})^{-1/2}$).
\begin{align}
&\eta(z):=\bm{:}\exp\bigg(-\sum_{n\neq{0}}(1-t^{n})a_{n}\frac{z^{-n}}{n}\bigg)\bm{:}, \\
&\xi(z):=\bm{:}\exp\bigg(\sum_{n\neq{0}}(1-t^{n})\gamma^{|n|}a_{n}\frac{z^{-n}}{n}\bigg)\bm{:}, \\
&\phi(z):=\exp\bigg(\sum_{n>0}\frac{1-t^{n}}{1-q^{n}}a_{-n}\frac{z^{n}}{n}\bigg).
\end{align}
We use the symbol $\phi_{N}(x):=\prod_{j=1}^{N}\phi(x_{j})$.

(1) The operator $\eta(z)$ reproduces the Macdonald operator $H_{N}(q,t)$ as follows.
\begin{align}
[\eta(z)]_{1}\phi_{N}(x)|0 \rangle=t^{-N}\{(t-1)H_{N}(q,t)+1\}\phi_{N}(x)|0 \rangle.
\end{align}

(2) The operator $\xi(z)$ reproduces the Macdonald operator $H_{N}(q^{-1},t^{-1})$ as follows.
\begin{align}
[\xi(z)]_{1}\phi_{N}(x)|0 \rangle=t^{N}\{(t^{-1}-1)H_{N}(q^{-1},t^{-1})+1\}\phi_{N}(x)|0 \rangle.
\end{align}
\end{prop}

\begin{proof}[\textit{Proof}]
We show the equation (2.6). First by Wick's theorem, we have
\begin{align*}
\eta(z)\phi(w)=\frac{1-w/z}{1-tw/z}\bm{:}\eta(z)\phi(w)\bm{:}
\end{align*}
and also have
\begin{align*}
\eta(z)\phi_{N}(x)=\prod_{j=1}^{N}\frac{1-x_{j}/z}{1-tx_{j}/z}\bm{:}\eta(z)\phi_{N}(x)\bm{:}.
\end{align*}
By the partial fraction expansion as
\begin{align*}
\prod_{j=1}^{N}\frac{1-x_{j}/z}{1-tx_{j}/z}=\frac{1-t}{1-t^{N}}\sum_{i=1}^{N}\frac{1-t^{-N+1}x_{i}/z}{1-tx_{i}/z}\prod_{j\neq{i}}\frac{tx_{i}-x_{j}}{x_{i}-x_{j}}
\end{align*}
and the formal expression of the delta function as
\begin{align*}
\delta(x):=\sum_{n\in\mathbf{Z}}x^{n}=\frac{1}{1-x}+\frac{x^{-1}}{1-x^{-1}},
\end{align*}
we have the following relation.
\begin{align*}
\prod_{j=1}^{N}\frac{1-x_{j}/z}{1-tx_{j}/z}=t^{-N}(t-1)\sum_{i=1}^{N}\prod_{j\neq{i}}\frac{tx_{i}-x_{j}}{x_{i}-x_{j}}\delta\Big(t\frac{x_{i}}{z}\Big)+t^{-N}\prod_{j=1}^{N}\frac{1-z/x_{j}}{1-t^{-1}z/x_{j}}.
\end{align*}
We use the notation $(\eta(z))_{\pm}$ defined by
\begin{align*}
(\eta(z))_{\pm}:=\exp\bigg(-\sum_{\pm n>0}(1-t^{n})a_{n}\frac{z^{-n}}{n}\bigg).
\end{align*}
Due to the relation $(\eta(tz))_{-}\phi(z)=\phi(qz)=T_{q,z}\phi(z)$, we have the following.
\begin{align*}
&\quad [\eta(z)]_{1}\phi_{N}(x)|0 \rangle \\
&=\left\{t^{-N}(t-1)\sum_{i=1}^{N}\prod_{j\neq{i}}\frac{tx_{i}-x_{j}}{x_{i}-x_{j}}(\eta(tx_{i}))_{-}+t^{-N}\left[\prod_{j=1}^{N}\frac{1-z/x_{j}}{1-t^{-1}z/x_{j}}(\eta(z))_{-}\right]_{1}\right\}\phi_{N}(x)|0 \rangle \\
&=t^{-N}(t-1)\sum_{i=1}^{N}\prod_{j\neq{i}}\frac{tx_{i}-x_{j}}{x_{i}-x_{j}}T_{q,x_{i}}\phi_{N}(x)|0 \rangle +t^{-N}\phi_{N}(x)|0 \rangle \\
&=t^{-N}\{(t-1)H_{N}(q,t)+1\}\phi_{N}(x)|0 \rangle,
\end{align*}
where we use the equation 
\begin{align*}
\left[\prod_{j=1}^{N}\frac{1-z/x_{j}}{1-t^{-1}z/x_{j}}(\eta(z))_{-}\right]_{1}=1.
\end{align*} 
The formula (2.7) is shown in the similar way. \quad $\Box$
\end{proof}

We also have the dual version of the proposition 2.1. Let $\langle 0|$ be the dual vacuum vector which satisfies the condition $\langle 0|a_{n}=0$ $(n<0)$ and define the dual boson Fock space $\mathcal{F}^{\ast}$ as a right $\mathcal{B}$ module :
\begin{align*}
\mathcal{F}^{\ast}:=\text{span}\{\langle 0|a_{\lambda} : \lambda\in\mathcal{P}\} \quad (a_{\lambda}:=a_{\lambda_{1}}\cdots a_{\ell(\lambda)}).
\end{align*}

\begin{prop}[\textbf{Dual version of the proposition 2.1}]
Let us define an operator $\phi^{\ast}(z) : \mathcal{F}^{\ast} \to \mathcal{F}^{\ast}\otimes\mathbf{C}[[z,z^{-1}]]$ as
\begin{align}
\phi^{\ast}(z):=\exp\bigg(\sum_{n>0}\frac{1-t^{n}}{1-q^{n}}a_{n}\frac{z^{n}}{n}\bigg).
\end{align}
We use the symbol $\phi^{\ast}_{N}(x):=\prod_{j=1}^{N}\phi^{\ast}(x_{j})$.

(1) The operator $\eta(z)$ reproduces the Macdonald operator $H_{N}(q,t)$ as follows.
\begin{align}
\langle 0|\phi^{\ast}_{N}(x)[\eta(z)]_{1}=t^{-N}\{(t-1)H_{N}(q,t)+1\}\langle 0|\phi^{\ast}_{N}(x).
\end{align}

(2) The operator $\xi(z)$ reproduces the Macdonald operator $H_{N}(q^{-1},t^{-1})$ as follows.
\begin{align}
\langle 0|\phi^{\ast}_{N}(x)[\xi(z)]_{1}=t^{N}\{(t^{-1}-1)H_{N}(q^{-1},t^{-1})+1\}\langle 0|\phi^{\ast}_{N}(x).
\end{align}
\end{prop}

\begin{remark}
The kernel function of the Macdonald operator is defined by [2][3]
\begin{align}
\Pi(q,t)(x,y):=\prod_{i,j}\frac{(tx_{i}y_{j};q)_{\infty}}{(x_{i}y_{j};q)_{\infty}}.
\end{align}
The free field realization of the Macdonald operator is based on the form of the kernel function $\Pi(q,t)(x,y)$.

The operators $\phi^{\ast}_{M}(x)$, $\phi_{N}(y)$ and the kernel function $\Pi(q,t)(x,y)$ satisfy the following.
\begin{align*}
\langle 0|\phi^{\ast}_{M}(x)\phi_{N}(y)|0 \rangle=\Pi(q,t)(\{x_{i}\}_{i=1}^{M},\{y_{j}\}_{j=1}^{N}):=\prod_{\genfrac{}{}{0pt}{1}{1\leq{i}\leq{M}}{1\leq{j}\leq{N}}}\frac{(tx_{i}y_{j};q)_{\infty}}{(x_{i}y_{j};q)_{\infty}}.
\end{align*}
\end{remark}

By the free field realization of the Macdonald operator, we can show the functional equation of the kernel function $\Pi(q,t)(\{x_{i}\}_{i=1}^{M},\{y_{j}\}_{j=1}^{N})$.

\begin{prop}[\textbf{Functional equation of the kernel function}]
The Macdonald operator $H_{N}(q,t)$ and the kernel function $\Pi(q,t)(\{x_{i}\}_{i=1}^{M},\{y_{j}\}_{j=1}^{N})$ satisfy the following functional equation.
\begin{align}
\{H_{M}(q,t)_{x}{-}t^{M-N}H_{N}(q,t)_{y}\}\Pi(q,t)(\{x_{i}\}_{i=1}^{M},\{y_{j}\}_{j=1}^{N})
=\frac{1{-}t^{M-N}}{1{-}t}\Pi(q,t)(\{x_{i}\}_{i=1}^{M},\{y_{j}\}_{j=1}^{N}).
\end{align}
Here $H_{M}(q,t)_{x}$ denotes the Macdonald operator which acts on functions of $x_{1},\cdots,x_{M}$.
\end{prop}

\begin{proof}[\textit{Proof}]
The proof of the proposition 2.4 is very simple. By the free field realization of the Macdonald operator, we can calculate the matrix element $\langle 0|\phi_{M}(x)[\eta(z)]_{1}\phi_{N}(y)|0 \rangle$ in different two ways as
\begin{align*}
&\quad \langle 0|\phi_{M}(x)[\eta(z)]_{1}\phi_{N}(y)|0 \rangle \\
&=t^{-M}\{(t-1)H_{M}(q,t)_{x}+1\}\Pi(q,t)(\{x_{i}\}_{i=1}^{M},\{y_{j}\}_{j=1}^{N}) \\
&=t^{-N}\{(t-1)H_{N}(q,t)_{y}+1\}\Pi(q,t)(\{x_{i}\}_{i=1}^{M},\{y_{j}\}_{j=1}^{N}).
\end{align*}
Consequently we have the relation (2.12). \quad $\Box$
\end{proof}

\begin{remark}
In [9], Komori, Noumi, and Shiraishi give another proof of the functional equation (2.12) which doesn't rely on the free field realization.
\end{remark}

\subsection{Ding-Iohara algebra $\mathcal{U}(q,t)$}
\begin{definition}[\textbf{Ding-Iohara algebra $\mathcal{U}(q,t)$ [8]}] 
Let us define the structure function $g(x)$ as
\begin{align}
g(x):=\frac{(1-qx)(1-t^{-1}x)(1-q^{-1}tx)}{(1-q^{-1}x)(1-tx)(1-qt^{-1}x)}.
\end{align}
Then $g(x^{-1})=g(x)^{-1}$. Let $\gamma$ be a central, invertible element and $x^{\pm}(z):=\sum_{n\in{\mathbf{Z}}}x^{\pm}_{n}z^{-n}$, $\psi^{\pm}(z):=\sum_{n\in\mathbf{Z}}\psi^{\pm}_{n}z^{-n}$ be currents satisfying the relations :
\begin{align}
&\hskip 1.5cm [\psi^{\pm}(z), \psi^{\pm}(w)]=0, \quad \psi^{+}(z)\psi^{-}(w)=\frac{g(\gamma z/w)}{g(\gamma^{-1}z/w)}\psi^{-}(w)\psi^{+}(z), \notag\\
&\psi^{\pm}(z)x^{+}(w)=g\left(\gamma^{\pm\frac{1}{2}}\frac{z}{w}\right)x^{+}(w)\psi^{\pm}(z), \quad \psi^{\pm}(z)x^{-}(w)=g\left(\gamma^{\mp\frac{1}{2}}\frac{z}{w}\right)^{-1}x^{-}(w)\psi^{\pm}(z), \notag\\
&\hskip 3cm x^{\pm}(z)x^{\pm}(w)=g\left(\frac{z}{w}\right)^{\pm 1}x^{\pm}(w)x^{\pm}(z), \notag\\
&[x^{+}(z),x^{-}(w)]=\frac{(1-q)(1-t^{-1})}{1-qt^{-1}}\bigg\{\delta\Big(\gamma\frac{w}{z}\Big)\psi^{+}\big(\gamma^{1/2}w\big)-\delta\Big(\gamma^{-1}\frac{w}{z}\Big)\psi^{-}\big(\gamma^{-1/2}w\big)\bigg\}.
\end{align}
We define the Ding-Iohara algebra $\mathcal{U}(q,t)$ to be an associative $\mathbf{C}$-algebra generated by $\{x^{\pm}_{n}\}_{n\in{\mathbf{Z}}},\,\{\psi^{\pm}_{n}\}_{n\in{\mathbf{Z}}}$, and $\gamma$.
\end{definition}

By Wick's theorem, we have the following proposition.

\begin{prop}[\textbf{Free field realization of the Ding-Iohara algebra $\mathcal{U}(q,t)$ [8]}] 
Set $\gamma:=(qt^{-1})^{-1/2}$ and define operators $\varphi^{\pm}(z) : \mathcal{F} \to \mathcal{F}\otimes\mathbf{C}[[z,z^{-1}]]$ as follows :
\begin{align}
\varphi^{+}(z):=\bm{:}\eta(\gamma^{1/2}z)\xi(\gamma^{-1/2}z)\bm{:}, \quad \varphi^{-}(z):=\bm{:}\eta(\gamma^{-1/2}z)\xi(\gamma^{1/2}z)\bm{:}.
\end{align}
Then the map
\begin{align*}
x^{+}(z) \mapsto \eta(z), \quad x^{-}(z) \mapsto \xi(z), \quad \psi^{\pm}(z) \mapsto \varphi^{\pm}(z)
\end{align*}
gives a representation of the Ding-Iohara algebra $\mathcal{U}(q,t)$.
\end{prop}

\subsection{Trigonometric Feigin-Odesskii algebra $\mathcal{A}$}
In this subsection, we review basic facts of the trigonometric Feigin-Odesskii algebra [8].

\begin{definition}[\textbf{Trigonometric Feigin-Odesskii algebra $\mathcal{A}$}] 
Let $\varepsilon_{n}(q;x) \,(n\in\mathbf{Z}_{>0})$ be a function defined as
\begin{align}
\varepsilon_{n}(q;x):=\prod_{1\leq{a}<b\leq{n}}\frac{(x_{a}-qx_{b})(x_{a}-q^{-1}x_{b})}{(x_{a}-x_{b})^{2}}.
\end{align}
We also define $\omega(x,y)$ as
\begin{align}
\omega(x,y):=\frac{(x-q^{-1}y)(x-ty)(x-qt^{-1}y)}{(x-y)^{3}}.
\end{align}
We define the action of the $N$-th symmetric group $\mathfrak{S}_{N}$ on $N$-variable functions by 
\begin{align*}
\sigma\cdot(f(x_{1},\cdots,x_{N})):=f(x_{\sigma(1)},\cdots,x_{\sigma(N)}) \quad (\sigma\in\mathfrak{S}_{N}). 
\end{align*}
We define the symmetrizer as
\begin{align}
\text{Sym}[f(x_{1},\cdots,x_{N})]:=\frac{1}{N!}\sum_{\sigma\in\mathfrak{S}_{N}}\sigma\cdot(f(x_{1},\cdots,x_{N})).
\end{align}
For a $m$-variable function $f(x_{1},\cdots,x_{m})$ and a $n$-variable function $g(x_{1},\cdots,x_{n})$, we define the star product $\ast$ as follows.
\begin{align}
(f\ast g)(x_{1},\cdots,x_{m+n}):=\text{Sym}\bigg[f(x_{1},\cdots,x_{m})g(x_{m+1},\cdots,x_{m+n})\prod_{\genfrac{}{}{0pt}{1}{1\leq{\alpha}\leq{m}}{m+1\leq{\beta}\leq{m+n}}}\omega(x_{\alpha},x_{\beta})\bigg].
\end{align}
For a partition $\lambda$, we define $\varepsilon_{\lambda}(q;x)$ as
\begin{align}
\varepsilon_{\lambda}(q;x):=\varepsilon_{\lambda_{1}}(q;x)\ast\cdots\ast\varepsilon_{\lambda_{\ell(\lambda)}}(q;x).
\end{align}

Set $\mathcal{A}_{0}:=\mathbf{Q}(q,t)$, $\mathcal{A}_{n}:=\text{span}\{\varepsilon_{\lambda}(q;x) : |\lambda|=n\} \,(n\geq{1})$. We define the trigonometric Feigin-Odesskii algebra $\mathcal{A}:=\bigoplus_{n\geq{0}}\mathcal{A}_{n}$ whose algebraic structure is given by the star product $\ast$. 
\end{definition}

\begin{remark}
The definition of the trigonometric Feigin-Odesskii algebra $\mathcal{A}$ above is a reduced version of the paper [8]. For instance, there would be a question why the function $\varepsilon_{n}(q;x)$ appears. For more detail of the trigonometric Feigin-Odesskii algebra $\mathcal{A}$, see [8].
\end{remark}

In the paper [8], the following fact is shown.

\begin{prop}[] 
The trigonometric Feigin-Odesskii algebra $(\mathcal{A},\ast)$ is unital, associative, and commutative.
\end{prop}

\subsection{Commutative families $\mathcal{M}$, $\mathcal{M}^{\prime}$}
Here we give an overview of the construction of the commutative families of the Macdonald operator by using the Ding-Iohara algebra and the trigonometric Feigin-Odesskii algebra.

\begin{definition}[\textbf{Map $\mathcal{O}$}] Define a linear map $\mathcal{O} : \mathcal{A} \to \text{End}(\mathcal{F})$ as
\begin{align}
\mathcal{O}(f):=\bigg[f(z_{1},\cdots,z_{n})\prod_{1\leq{i<j}\leq{n}}\omega(z_{i},z_{j})^{-1}\eta(z_{1})\cdots\eta(z_{n})\bigg]_{1} \quad(f\in\mathcal{A}_{n}).
\end{align}
Here $[f(z_{1},\cdots,z_{n})]_{1}$ denotes the constant term of $f(z_{1},\cdots,z_{n})$ in $z_{1},\cdots,z_{n}$, and we extend the map $\mathcal{O}$ linearly.
\end{definition}

\begin{prop}[] 
The map $\mathcal{O}$ and the star product $\ast$ are compatible : for $f,g\in\mathcal{A}$, we have $\mathcal{O}(f\ast g)=\mathcal{O}(f)\mathcal{O}(g)$.
\end{prop}

\begin{proof}[\textit{Proof}] 
To prove the proposition, for $f\in\mathcal{A}_{m}$ and $g\in\mathcal{A}_{n}$, we show $\mathcal{O}(f\ast g)=\mathcal{O}(f)\mathcal{O}(g)$. First we have
\begin{align*}
&\quad \mathcal{O}(f\ast g)(x_{1},\cdots,x_{m+n}) \\
&=\bigg[\text{Sym}\bigg(f(x_{1},\cdots,x_{m})g(x_{m+1},\cdots,x_{m+n})\prod_{\genfrac{}{}{0pt}{1}{1\leq{\alpha}\leq{m}}{m+1\leq{\beta}\leq{m+n}}}\omega(x_{\alpha},x_{\beta})\bigg) \\
&\hskip 5cm \times \prod_{1\leq{i<j}\leq{m+n}}\omega(x_{i},x_{j})^{-1}\eta(x_{1})\cdots\eta(x_{m+n})\bigg]_{1}.
\end{align*}
Then from the relation 
\begin{align*}
\eta(z)\eta(w)=g\Big(\frac{z}{w}\Big)\eta(w)\eta(z),
\end{align*}
we have the following :
\begin{align}
\frac{1}{\omega(z,w)}\eta(z)\eta(w)=\frac{1}{\omega(w,z)}\eta(w)\eta(z).
\end{align}
The relation shows that the operator-valued function 
\begin{align}
\prod_{1\leq{i<j}\leq{N}}\omega(x_{i},x_{j})^{-1}\eta(x_{1})\cdots\eta(x_{N})
\end{align}
is symmetric in $x_{1},\cdots,x_{N}$. Further we have
\begin{align*}
[\text{Sym}(F(x_{1},\cdots,x_{N}))]_{1}=[F(x_{1},\cdots,x_{N})]_{1}.
\end{align*}
This follows from the fact that constant terms are invariant under the action of the symmetric group. In addition for a symmetric function $f(x_{1},\cdots,x_{N})$, we have
\begin{align*}
\sigma(f(x_{1},\cdots,x_{N})g(x_{1},\cdots,x_{N}))=f(x_{1},\cdots,x_{N})\sigma(g(x_{1},\cdots,x_{N})) \quad (\sigma\in{\mathfrak{S}_{N}}).
\end{align*}
Hence we have
\begin{align*}
\text{Sym}(f(x_{1},\cdots,x_{N})g(x_{1},\cdots,x_{N}))=f(x_{1},\cdots,x_{N})\text{Sym}(g(x_{1},\cdots,x_{N})).
\end{align*}
From them we have the following.
\begin{align*}
\mathcal{O}(f\ast g)
&=\bigg[\text{Sym}\bigg(f(x_{1},\cdots,x_{m})g(x_{m+1},\cdots,x_{m+n})\prod_{\genfrac{}{}{0pt}{1}{1\leq{\alpha}\leq{m}}{m+1\leq{\beta}\leq{m+n}}}\omega(x_{\alpha},x_{\beta}) \\
&\hskip 5cm \times\prod_{1\leq{i<j}\leq{m+n}}\omega(x_{i},x_{j})^{-1}\eta(x_{1})\cdots\eta(x_{m+n})\bigg)\bigg]_{1} \\
&=\bigg[f(x_{1},\cdots,x_{m})g(x_{m+1},\cdots,x_{m+n})\prod_{\genfrac{}{}{0pt}{1}{1\leq{\alpha}\leq{m}}{m+1\leq{\beta}\leq{m+n}}}\omega(x_{\alpha},x_{\beta}) \\
&\hskip 5cm \times\prod_{1\leq{i<j}\leq{m+n}}\omega(x_{i},x_{j})^{-1}\eta(x_{1})\cdots\eta(x_{m+n})\bigg]_{1} \\
&=\bigg[f(x_{1},\cdots,x_{m})\prod_{1\leq{i<j}\leq{m}}\omega(x_{i},x_{j})^{-1}\eta(x_{1})\cdots\eta(x_{m}) \\
&\hskip 1cm \times g(x_{m+1},\cdots,x_{m+n})\prod_{m+1\leq{i<j}\leq{m+n}}\omega(x_{i},x_{j})^{-1}\eta(x_{m+1})\cdots\eta(x_{m+n})\bigg]_{1} \\
&=\mathcal{O}(f)\mathcal{O}(g). \quad \qed
\end{align*}
\end{proof}

The trigonometric Feigin-Odesskii algebra $\mathcal{A}$ is commutative by means of the star product $\ast$, therefore we have the following corollary.

Let $V$ be a $\mathbf{C}$-vector space and $T : V \to V$ be a $\mathbf{C}$-linear operator. Then for a subset $W\subset V$, the symbol $T|_{W}$ denotes the restriction of $T$ on $W$. For a subset $M\subset\text{End}_{\mathbf{C}}(V)$, we use the symbol $M|_{W}:=\{T|_{W} : T\in M\}$ $(W\subset V)$.

\begin{cor}[\textbf{Commutative family $\mathcal{M}$}] 
(1) Set $\mathcal{M}:=\mathcal{O}(\mathcal{A})$. The space $\mathcal{M}$ consists of operators commuting with each other : $[\mathcal{O}(f),\mathcal{O}(g)]=0$ $(f,g\in\mathcal{A})$. 

(2) The space $\mathcal{M}|_{\mathbf{C}\phi_{N}(x)|0 \rangle}$ is a set of commuting $q$-difference operators which contains the Macdonald operator $H_{N}(q,t)$ (commutative family of the Macdonald operator $H_{N}(q,t)$).
\end{cor}

\begin{proof}[\textit{Proof}]
(1) This statement follows from the commutativity of $\mathcal{A}$ and the compatibility of the star product $\ast$ and the map $\mathcal{O}$. We have $\mathcal{O}(f\ast g)=\mathcal{O}(f)\mathcal{O}(g)$ and also have $\mathcal{O}(f\ast g)=\mathcal{O}(g\ast f)=\mathcal{O}(g)\mathcal{O}(f)$. This shows $[\mathcal{O}(f),\mathcal{O}(g)]=0$. 

(2) Due to the free field realization of the Macdonald operator $H_{N}(q,t)$, the operator $\mathcal{O}(\varepsilon_{r}(q;z))$ $(r\in\mathbf{Z}_{>0})$ acts on $\phi_{N}(x)|0 \rangle$ $(N\in\mathbf{Z}_{>0})$ as a $r$-th order $q$-difference operator. By the fact that $\mathcal{M}=\mathcal{O}(\mathcal{A})$ is generated by $\{\mathcal{O}(\varepsilon_{r}(q;z))\}_{r\in\mathbf{Z}_{>0}}$ and the relation
\begin{align*}
\mathcal{O}(\varepsilon_{1}(q;z))\phi_{N}(x)|0 \rangle=[\eta(z)]_{1}\phi_{N}(x)|0 \rangle=t^{-N}\{(t-1)H_{N}(q,t)+1\}\phi_{N}(x)|0 \rangle
\end{align*}
and (1) in the corollary 2.13, the restriction $\mathcal{M}|_{\mathbf{C}\phi_{N}(x)|0 \rangle}$ is a space of commuting $q$-difference operators which contains the Macdonald operator $H_{N}(q,t)$. \quad $\Box$
\end{proof}

The Macdonald operator $H_{N}(q^{-1},t^{-1})$ is reproduced from the operator $\xi(z)$. By this fact, we can construct another commutative family of the Macdonald operator.

\begin{definition}[\textbf{Map $\mathcal{O}^{\prime}$}] Set a function $\omega^{\prime}(x,y)$ as follows.
\begin{align}
\omega^{\prime}(x,y):=\frac{(x-qy)(x-t^{-1}y)(x-q^{-1}ty)}{(x-y)^{3}}.
\end{align}
Define a linear map $\mathcal{O}^{\prime} : \mathcal{A} \to \text{End}(\mathcal{F})$ as
\begin{align}
\mathcal{O}^{\prime}(f):=\bigg[f(z_{1},\cdots,z_{n})\prod_{1\leq{i<j}\leq{n}}\omega^{\prime}(z_{i},z_{j})^{-1}\xi(z_{1})\cdots\xi(z_{n})\bigg]_{1} \quad(f\in\mathcal{A}_{n}).
\end{align}
We extend the map $\mathcal{O}^{\prime}$ linearly.
\end{definition}

\begin{lemma}
Define another star product $\ast^{\prime}$ as follows :
\begin{align*}
(f\ast^{\prime} g)(x_{1},\cdots,x_{m+n}):=\text{\rm Sym}\bigg[f(x_{1},\cdots,x_{m})g(x_{m+1},\cdots,x_{m+n})\prod_{\genfrac{}{}{0pt}{1}{1\leq{\alpha}\leq{m}}{m+1\leq{\beta}\leq{m+n}}}\omega^{\prime}(x_{\alpha},x_{\beta})\bigg].
\end{align*}
In the trigonometric Feigin-Odesskii algebra $\mathcal{A}$, we have $\ast^{\prime}=\ast$.
\end{lemma}

\begin{proof}[\textit{Proof}]
From the relation $\omega(x,y)=\omega^{\prime}(y,x)$, for $f\in\mathcal{A}_{m}$, $g\in\mathcal{A}_{n}$ we have the following.
\begin{align*}
(f\ast^{\prime} g)(x_{1},\cdots,x_{m+n})&=\text{Sym}\bigg[f(x_{1},\cdots,x_{m})g(x_{m+1},\cdots,x_{m+n})\prod_{\genfrac{}{}{0pt}{1}{1\leq{\alpha}\leq{m}}{m+1\leq{\beta}\leq{m+n}}}\omega^{\prime}(x_{\alpha},x_{\beta})\bigg] \\
&=\text{Sym}\bigg[f(x_{1},\cdots,x_{m})g(x_{m+1},\cdots,x_{m+n})\prod_{\genfrac{}{}{0pt}{1}{1\leq{\alpha}\leq{m}}{m+1\leq{\beta}\leq{m+n}}}\omega(x_{\beta},x_{\alpha})\bigg] \\
&=\text{Sym}\bigg[g(x_{1},\cdots,x_{n})f(x_{n+1},\cdots,x_{n+m})\prod_{\genfrac{}{}{0pt}{1}{1\leq{\alpha}\leq{n}}{n+1\leq{\beta}\leq{n+m}}}\omega(x_{\alpha},x_{\beta})\bigg] \\
&=(g \ast f)(x_{1},\cdots,x_{m+n}) \\
&=(f\ast g)(x_{1},\cdots,x_{m+n}) \quad (\because \, \text{$\mathcal{A}$ is commutative by means of $\ast$}).
\end{align*}
Since the relation holds for any $f\in\mathcal{A}_{m}$, $g\in\mathcal{A}_{n}$, we have $\ast^{\prime}=\ast$. \quad $\Box$
\end{proof}

We can check the map $\mathcal{O}^{\prime}$ and the star product $\ast^{\prime}$ are compatible in the similar way of the proof of the proposition 2.12. Furthermore by the lemma 2.15 as $\ast^{\prime}=\ast$, we have the following corollary.

\begin{cor}[\textbf{Commutative family $\mathcal{M}^{\prime}$}]
(1) Set $\mathcal{M}^{\prime}:=\mathcal{O}^{\prime}(\mathcal{A})$. The space $\mathcal{M}^{\prime}$ consists of operators commuting with each other. 

(2) The space $\mathcal{M}^{\prime}|_{\mathbf{C}\phi_{N}(x)|0 \rangle}$ is a set of commuting $q$-difference operators which contains the Macdonald operator $H_{N}(q^{-1},t^{-1})$ (commutative family of the Macdonald operator $H_{N}(q^{-1},t^{-1})$).
\end{cor}

From the relation $[[\eta(z)]_{1},[\xi(w)]_{1}]=0$, we have the following proposition.

\begin{prop}
The commutative families $\mathcal{M}$, $\mathcal{M}^{\prime}$ satisfy $[\mathcal{M},\mathcal{M}^{\prime}]=0$ : For any $a\in\mathcal{M}$ and $a^{\prime}\in\mathcal{M}^{\prime}$, we have $[a,a^{\prime}]=0$.
\end{prop}

\begin{proof}[\textit{Proof}]
This proposition follows from the existence of the Macdonald symmetric functions. That is, elements of the commutative families are simultaneously diagonalized by the Macdonald symmetric functions. \quad $\Box$
\end{proof}

From the proposition 2.17, commutative families $\mathcal{M}|_{\mathbf{C}\phi_{N}(x)|0 \rangle}$, $\mathcal{M}^{\prime}|_{\mathbf{C}\phi_{N}(x)|0 \rangle}$ also commute with each other : $[\mathcal{M}|_{\mathbf{C}\phi_{N}(x)|0 \rangle}, \mathcal{M}^{\prime}|_{\mathbf{C}\phi_{N}(x)|0 \rangle}]=0$.

\section{Elliptic theory side}
In this section, we introduce the elliptic Ding-Iohara algebra [14] and explain how commutative families of the elliptic Macdonald operator are constructed [15].

\subsection{Free field realization of the elliptic Macdonald operator}
The elliptic Macdonald operator $H_{N}(q,t,p)$ $(N\in\mathbf{Z}_{>0})$ is defined by
\begin{align}
H_{N}(q,t,p):=\sum_{i=1}^{N}\prod_{j\neq{i}}\frac{\Theta_{p}(tx_{i}/x_{j})}{\Theta_{p}(x_{i}/x_{j})}T_{q,x_{i}}.
\end{align}

Set an algebra of boson $\mathcal{B}_{a,b}$ generated by $\{a_{n}\}_{n\in{\mathbf{Z}\setminus\{0\}}}$, $\{b_{n}\}_{n\in{\mathbf{Z}\setminus\{0\}}}$ and the following relations :
\begin{align*}
&[a_{m},a_{n}]=m(1-p^{|m|})\frac{1-q^{|m|}}{1-t^{|m|}}\delta_{m+n,0}, \quad [b_{m},b_{n}]=m\frac{1-p^{|m|}}{(qt^{-1}p)^{|m|}}\frac{1-q^{|m|}}{1-t^{|m|}}\delta_{m+n,0}, \\
&[a_{m},b_{n}]=0.
\end{align*}
Let $|0 \rangle$ be the vacuum vector which satisfies the condition $a_{n}|0 \rangle=b_{n}|0 \rangle=0$ $(n>0)$ and set the boson Fock space $\mathcal{F}$ as a left $\mathcal{B}_{a,b}$ module.
\begin{align*}
\mathcal{F}=\text{\rm span}\{a_{-\lambda}b_{-\mu}|0 \rangle : \lambda, \mu\in{\mathcal{P}}\}.
\end{align*}
We also define the normal ordering $\bm{:} \bullet \bm{:}$ as usual :
\begin{align*}
\bm{:}a_{m}a_{n}\bm{:}=
\begin{cases}
a_{m}a_{n} \quad (m<n), \\
a_{n}a_{m} \quad (m\geq{n}),
\end{cases}
\bm{:}b_{m}b_{n}\bm{:}=
\begin{cases}
b_{m}b_{n} \quad (m<n), \\
b_{n}b_{m} \quad (m\geq{n}).
\end{cases}
\end{align*}

\begin{theorem}[\textbf{Free field realization of the elliptic Macdonald operator [14]}]
Define the operators $\eta(p;z)$, $\xi(p;z)$ as follows $(\gamma:=(qt^{-1})^{-1/2})$.
\begin{align}
&\eta(p;z):=\bm{:}\exp\bigg(-\sum_{n\neq{0}}\frac{1-t^{-n}}{1-p^{|n|}}p^{|n|}b_{n}\frac{z^{n}}{n}\bigg)\exp\bigg(-\sum_{n\neq{0}}\frac{1-t^{n}}{1-p^{|n|}}a_{n}\frac{z^{-n}}{n}\bigg)\bm{:}, \\
&\xi(p;z)
:=\bm{:}\exp\bigg(\sum_{n\neq{0}}\frac{1-t^{-n}}{1-p^{|n|}}\gamma^{-|n|}p^{|n|}b_{n}\frac{z^{n}}{n}\bigg)\exp\bigg(\sum_{n\neq{0}}\frac{1-t^{n}}{1-p^{|n|}}\gamma^{|n|}a_{n}\frac{z^{-n}}{n}\bigg)\bm{:}.
\end{align}
Let $\phi(p;z) : \mathcal{F} \to \mathcal{F}\otimes\mathbf{C}[[z,z^{-1}]]$ be an operator defined as follows.
\begin{align}
\phi(p;z):=\exp\bigg(\sum_{n>0}\frac{(1-t^{n})(qt^{-1}p)^{n}}{(1-q^{n})(1-p^{n})}b_{-n}\frac{z^{-n}}{n}\bigg)\exp\bigg(\sum_{n>0}\frac{1-t^{n}}{(1-q^{n})(1-p^{n})}a_{-n}\frac{z^{n}}{n}\bigg).
\end{align}
We use the symbol $\phi_{N}(p;x):=\prod_{j=1}^{N}\phi(p;x_{j})$.

(1) The elliptic Macdonald operator $H_{N}(q,t,p)$ is reproduced by the operator $\eta(p;z)$ as follows.
\begin{align}
[\eta(p;z){-}t^{-N}(\eta(p;z))_{-}(\eta(p;p^{-1}z))_{+}]_{1}\phi_{N}(p;x)|0 \rangle 
=\frac{t^{-N+1}\Theta_{p}(t^{-1})}{(p;p)_{\infty}^{3}}H_{N}(q,t,p)\phi_{N}(p;x)|0 \rangle.
\end{align}
Here we use the notation $(\eta(p;z))_{\pm}$ as
\begin{align*}
(\eta(p;z))_{\pm}:=\exp\bigg(-\sum_{\pm n>0}\frac{1-t^{-n}}{1-p^{|n|}}p^{|n|}b_{n}\frac{z^{n}}{n}\bigg)\exp\bigg(-\sum_{\pm n>0}\frac{1-t^{n}}{1-p^{|n|}}a_{n}\frac{z^{-n}}{n}\bigg).
\end{align*}

(2) The elliptic Macdonald operator $H_{N}(q^{-1},t^{-1},p)$ is reproduced by the operator $\xi(p;z)$ as follows.
\begin{align}
[\xi(p;z){-}t^{N}(\xi(p;z))_{-}(\xi(p;p^{-1}z))_{+}]_{1}\phi_{N}(p;x)|0 \rangle 
=\frac{t^{N-1}\Theta_{p}(t)}{(p;p)_{\infty}^{3}}H_{N}(q^{-1},t^{-1},p)\phi_{N}(p;x)|0 \rangle.
\end{align}
Here we use the notation $(\xi(p;z))_{\pm}$ as
\begin{align*}
(\xi(p;z))_{\pm}:=\exp\bigg(\sum_{\pm n>0}\frac{1-t^{-n}}{1-p^{|n|}}\gamma^{-|n|}p^{|n|}b_{n}\frac{z^{n}}{n}\bigg)\exp\bigg(\sum_{\pm n>0}\frac{1-t^{n}}{1-p^{|n|}}\gamma^{|n|}a_{n}\frac{z^{-n}}{n}\bigg).
\end{align*}
\end{theorem}

\begin{proof}[\textit{Proof}]
Here we prove the formula (3.5). First we have
\begin{align*}
\eta(p;z)\phi(p;w)=\frac{\Theta_{p}(w/z)}{\Theta_{p}(tw/z)}\bm{:}\eta(p;z)\phi(p;w)\bm{:}
\end{align*}
and also have
\begin{align*}
\eta(p;z)\phi_{N}(p;x)=\prod_{j=1}^{N}\frac{\Theta_{p}(x_{j}/z)}{\Theta_{p}(tx_{j}/z)}\bm{:}\eta(p;z)\phi_{N}(p;x)\bm{:}.
\end{align*}
From the partial fraction expansion of formula a product of the theta function
\begin{align*}
\prod_{j=1}^{N}\frac{\Theta_{p}(x_{j}/z)}{\Theta_{p}(tx_{j}/z)}=\frac{\Theta_{p}(t)}{\Theta_{p}(t^{N})}\sum_{i=1}^{N}\frac{\Theta_{p}(t^{-N+1}x_{i}z)}{\Theta_{p}(tx_{i}/z)}\prod_{j\neq{i}}\frac{\Theta_{p}(tx_{i}/x_{j})}{\Theta_{p}(x_{i}/x_{j})}
\end{align*}
and the relation of the theta function and the delta function as
\begin{align*}
\frac{1}{\Theta_{p}(x)}+\frac{x^{-1}}{\Theta_{p}(x^{-1})}=\frac{1}{(p;p)_{\infty}^{3}}\delta(x),
\end{align*}
we have the following :
\begin{align*}
\prod_{j=1}^{N}\frac{\Theta_{p}(x_{j}/z)}{\Theta_{p}(tx_{j}/z)}=\frac{t^{-N+1}\Theta_{p}(t^{-1})}{(p;p)_{\infty}^{3}}\sum_{i=1}^{N}\prod_{j\neq{i}}\frac{\Theta_{p}(tx_{i}/x_{j})}{\Theta_{p}(x_{i}/x_{j})}\delta\Big( t\frac{x_{i}}{z}\Big)+t^{-N}\prod_{j=1}^{N}\frac{\Theta_{p}(px_{j}/z)}{\Theta_{p}(ptx_{j}/z)}.
\end{align*}
Here we use the relation $\Theta_{p}(ax)/\Theta_{p}(bx)=\Theta_{p}(pa^{-1}x^{-1})/\Theta_{p}(pb^{-1}x^{-1})$. Hence we have
\begin{align*}
&\quad [\eta(p;z)-t^{-N}(\eta(p;z))_{-}(\eta(p;p^{-1}z))_{+}]_{1}\phi_{N}(p;x)|0 \rangle \\
&=\frac{t^{-N+1}\Theta_{p}(t^{-1})}{(p;p)_{\infty}^{3}}\sum_{i=1}^{N}\prod_{j\neq{i}}\frac{\Theta_{p}(tx_{i}/x_{j})}{\Theta_{p}(x_{i}/x_{j})}(\eta(p;tx_{i}))_{-}\phi_{N}(p;x)|0 \rangle.
\end{align*}
By the relation $(\eta(p;tz))_{-}\phi(p;z)=\phi(p;qz)$, we have
\begin{align*}
&\quad [\eta(p;z)-t^{-N}(\eta(p;z))_{-}(\eta(p;p^{-1}z))_{+}]_{1}\phi_{N}(p;x)|0 \rangle \\
&=\frac{t^{-N+1}\Theta_{p}(t^{-1})}{(p;p)_{\infty}^{3}}H_{N}(q,t,p)\phi_{N}(p;x)|0 \rangle. \quad \Box
\end{align*}
\end{proof}

The theorem 3.1 is also stated as follows. First we set zero mode generators $a_{0}$, $Q$ which satisfy the relation :
\begin{align}
[a_{0},Q]=1, \quad [a_{n},a_{0}]=[b_{n},a_{0}]=0, \quad [a_{n},Q]=[b_{n},Q]=0 \quad (n\in\mathbf{Z}\setminus\{0\}).
\end{align}
We also set the condition $a_{0}|0 \rangle=0$. We define $|\alpha \rangle:=e^{\alpha Q}|0 \rangle$ $(\alpha\in\mathbf{C})$. Then we have $a_{0}|\alpha \rangle=\alpha|\alpha \rangle$. We also set  $\mathcal{F}_{\alpha}:=\text{span}\{a_{-\lambda}b_{-\mu}|\alpha \rangle : \lambda,\mu\in\mathcal{P}\}$ $(\alpha\in\mathbf{C})$.

\begin{theorem}[\textbf{[14]}]
Set the operators $\widetilde{\eta}(p;z)$, $\widetilde{\xi}(p;z)$ by
\begin{align*}
\widetilde{\eta}(p;z):=(\eta(p;z))_{-}(\eta(p;p^{-1}z))_{+}, \quad \widetilde{\xi}(p;z):=(\xi(p;z))_{-}(\xi(p;p^{-1}z))_{+}. 
\end{align*}
Using these symbols we define operators $E(p;z)$, $F(p;z)$ as follows :
\begin{align}
E(p;z):=\eta(p;z)-\widetilde{\eta}(p;z)t^{-a_{0}}, \quad F(p;z):=\xi(p;z)-\widetilde{\xi}(p;z)t^{a_{0}}.
\end{align}
Then the elliptic Macdonald operators $H_{N}(q,t,p)$, $H_{N}(q^{-1},t^{-1},p)$ are reproduced by the operators $E(p;z)$, $F(p;z)$ as follows.
\begin{align}
&[E(p;z)]_{1}\phi_{N}(p;x)|N\rangle=\frac{t^{-N+1}\Theta_{p}(t^{-1})}{(p;p)_{\infty}^{3}}H_{N}(q,t,p)\phi_{N}(p;x)|N\rangle, \\
&[F(p;z)]_{1}\phi_{N}(p;x)|N\rangle=\frac{t^{N-1}\Theta_{p}(t)}{(p;p)_{\infty}^{3}}H_{N}(q^{-1},t^{-1},p)\phi_{N}(p;x)|N\rangle.
\end{align}
\end{theorem}

The dual versions of the theorem 3.1, 3.2 are also available. Let $\langle 0|$ be the dual vacuum vector which satisfies the condition $\langle 0|a_{n}=\langle 0|b_{n}=0$ $(n<0)$ and $\langle 0|a_{0}=0$. We define the dual boson Fock space as a right $\mathcal{B}_{a,b}$ module :
\begin{align*}
\mathcal{F}^{\ast}:=\text{span}\{\langle 0|a_{\lambda}b_{\mu} : \lambda,\mu\in\mathcal{P}\}.
\end{align*}
For a complex number $\alpha\in\mathbf{C}$, set $\langle \alpha|:=\langle 0|e^{-\alpha Q}$. Then we have $\langle \alpha|a_{0}=\alpha\langle \alpha|$. We also set $\mathcal{F}^{\ast}_{\alpha}:=\text{span}\{\langle \alpha|a_{\lambda}b_{\mu} : \lambda, \mu\in\mathcal{P}\}$ $(\alpha\in\mathbf{C})$.

\begin{theorem}[\textbf{Dual versions of the theorem 3.1, 3.2}]
Let us define an operator $\phi^{\ast}(p;z) : \mathcal{F}^{\ast} \to \mathcal{F}^{\ast}\otimes\mathbf{C}[[z,z^{-1}]]$ as follows.
\begin{align}
\phi^{\ast}(p;z):=\exp\bigg(\sum_{n>0}\frac{(1-t^{n})(qt^{-1}p)^{n}}{(1-q^{n})(1-p^{n})}b_{n}\frac{z^{-n}}{n}\bigg)\exp\bigg(\sum_{n>0}\frac{1-t^{n}}{(1-q^{n})(1-p^{n})}a_{n}\frac{z^{n}}{n}\bigg).
\end{align}
We use the symbol $\phi^{\ast}_{N}(p;x):=\prod_{j=1}^{N}\phi^{\ast}(p;x_{j})$.

(1) The operators $\eta(p;z)$, $\xi(p;z)$ reproduce the elliptic Macdonald operators $H_{N}(q,t,p)$, $H_{N}(q^{-1},t^{-1},p)$ as follows.
\begin{align}
&\langle 0|\phi^{\ast}_{N}(p;x)[\eta(p;z){-}t^{-N}(\eta(p;z))_{-}(\eta(p;p^{-1}z))_{+}]_{1}
=\frac{t^{-N+1}\Theta_{p}(t^{-1})}{(p;p)_{\infty}^{3}}H_{N}(q,t,p)\langle 0|\phi^{\ast}_{N}(p;x), \\
&\langle 0|\phi^{\ast}_{N}(p;x)[\xi(p;z){-}t^{N}(\xi(p;z))_{-}(\xi(p;p^{-1}z))_{+}]_{1} 
=\frac{t^{N-1}\Theta_{p}(t)}{(p;p)_{\infty}^{3}}H_{N}(q^{-1},t^{-1},p)\langle 0|\phi^{\ast}_{N}(p;x).
\end{align} 

(2) The operators $E(p;z)$, $F(p;z)$ reproduce the elliptic Macdonald operators $H_{N}(q,t,p)$, $H_{N}(q^{-1},t^{-1},p)$ as follows.
\begin{align}
&\langle N|\phi^{\ast}_{N}(p;x)[E(p;z)]_{1}=\frac{t^{-N+1}\Theta_{p}(t^{-1})}{(p;p)_{\infty}^{3}}H_{N}(q,t,p)\langle N|\phi^{\ast}_{N}(p;x), \\
&\langle N|\phi^{\ast}_{N}(p;x)[F(p;z)]_{1}=\frac{t^{N-1}\Theta_{p}(t)}{(p;p)_{\infty}^{3}}H_{N}(q^{-1},t^{-1},p)\langle N|\phi^{\ast}_{N}(p;x).
\end{align}
\end{theorem}

\begin{remark}
The operator $\phi^{\ast}_{M}(p;x)$, $\phi_{N}(p;y)$ and the kernel function $\Pi(q,t,p)(x,y)$ satisfy the following.
\begin{align*}
\langle 0|\phi^{\ast}_{M}(p;x)\phi_{N}(p;y)|0 \rangle=\Pi(q,t,p)(\{x_{i}\}_{i=1}^{M},\{y_{j}\}_{j=1}^{N}):=\prod_{\genfrac{}{}{0pt}{1}{1\leq{i}\leq{M}}{1\leq{j}\leq{N}}}\frac{\Gamma_{q,p}(x_{i}y_{j})}{\Gamma_{q,p}(tx_{i}y_{j})}.
\end{align*} 
\end{remark}

By using the free field realization of the elliptic Macdonald operator, we obtain the functional equation of the elliptic kernel function $\Pi(q,t,p)(\{x_{i}\}_{i=1}^{M},\{y_{j}\}_{j=1}^{N})$ as follows.

\begin{theorem}[\textbf{Functional equation of the elliptic kernel function}]
Define the elliptic kernel function $\Pi(q,t,p)(\{x_{i}\}_{i=1}^{M},\{y_{j}\}_{j=1}^{N})$ by
\begin{align}
\Pi(q,t,p)(\{x_{i}\}_{i=1}^{M},\{y_{j}\}_{j=1}^{N}):=\prod_{\genfrac{}{}{0pt}{1}{1\leq{i}\leq{M}}{1\leq{j}\leq{N}}}\frac{\Gamma_{q,p}(x_{i}y_{j})}{\Gamma_{q,p}(tx_{i}y_{j})}.
\end{align}
We also define $C_{MN}(p;x,y)$ as 
\begin{align}
C_{MN}(p;x,y)&:=\frac{\langle 0|\phi^{\ast}_{M}(p;x)[(\eta(p;z))_{-}(\eta(p;p^{-1}z))_{+}]_{1}\phi_{N}(p;y)|0 \rangle}{\Pi(q,t,p)(\{x_{i}\}_{i=1}^{M},\{y_{j}\}_{j=1}^{N})} \notag\\
&=\left[\prod_{i=1}^{M}\frac{\Theta_{p}(t^{-1}x_{i}z)}{\Theta_{p}(x_{i}z)}\prod_{j=1}^{N}\frac{\Theta_{p}(z/y_{j})}{\Theta_{p}(t^{-1}z/y_{j})}\right]_{1}.
\end{align}
For the elliptic Macdonald operator and the function $\Pi(q,t,p)(\{x_{i}\}_{i=1}^{M},\{y_{j}\}_{j=1}^{N})$, we have the following functional equation.
\begin{align}
&\quad \{H_{M}(q,t,p)_{x}-t^{M-N}H_{N}(q,t,p)_{y}\}\Pi(q,t,p)(\{x_{i}\}_{i=1}^{M},\{y_{j}\}_{j=1}^{N}) \notag\\
&=\frac{(1-t^{M-N})(p;p)_{\infty}^{3}}{\Theta_{p}(t)}C_{MN}(p;x,y)\Pi(q,t,p)(\{x_{i}\}_{i=1}^{M},\{y_{j}\}_{j=1}^{N}).
\end{align}
Here the symbol $H_{M}(q,t,p)_{x}$ denotes the elliptic Macdonald operator which acts on functions of $x_{1},\cdots, x_{M}$.
\end{theorem}

\begin{proof}[\textit{Proof}]
The proof is straightforward. What we have to do is to calculate the matrix element $\langle 0|\phi^{\ast}_{M}(p;x)[\eta(p;z)]_{1}\phi_{N}(p;y)|0 \rangle$ by the theorem 3.1, 3.3 in different two ways as follows :
\begin{align*}
&\quad \langle 0|\phi^{\ast}_{M}(p;x)[\eta(p;z)]_{1}\phi_{N}(p;y)|0 \rangle \\
&=\frac{t^{-M+1}\Theta_{p}(t^{-1})}{(p;p)_{\infty}^{3}}H_{M}(q,t,p)_{x}\Pi(q,t,p)(\{x_{i}\}_{i=1}^{M},\{y_{j}\}_{j=1}^{N}) \\
&\hskip 4cm +t^{-M}C_{MN}(p;x,y)\Pi(q,t,p)(\{x_{i}\}_{i=1}^{M},\{y_{j}\}_{j=1}^{N}) \\
&=\frac{t^{-N+1}\Theta_{p}(t^{-1})}{(p;p)_{\infty}^{3}}H_{N}(q,t,p)_{y}\Pi(q,t,p)(\{x_{i}\}_{i=1}^{M},\{y_{j}\}_{j=1}^{N}) \\
&\hskip 4cm +t^{-N}C_{MN}(p;x,y)\Pi(q,t,p)(\{x_{i}\}_{i=1}^{M},\{y_{j}\}_{j=1}^{N}).
\end{align*}
Therefore we obtain the theorem 3.5. \quad $\Box$
\end{proof}

\begin{remark}
We can check the following :
\begin{align*}
C_{MN}(p;x,y)=&\left[\prod_{i=1}^{M}\frac{\Theta_{p}(t^{-1}x_{i}z)}{\Theta_{p}(x_{i}z)}\prod_{j=1}^{N}\frac{\Theta_{p}(z/y_{j})}{\Theta_{p}(t^{-1}z/y_{j})}\right]_{1} \\
\xrightarrow[p \to 0]{} &\left[\prod_{i=1}^{M}\frac{1-t^{-1}x_{i}z}{1-x_{i}z}\prod_{j=1}^{N}\frac{1-z/y_{j}}{1-t^{-1}z/y_{j}}\right]_{1}=1.
\end{align*}
Hence by taking the limit $p \to 0$ the formula (3.18) reduces to the equation (2.12).
\end{remark}

\subsection{Elliptic Ding-Iohara algebra $\mathcal{U}(q,t,p)$}
The elliptic Ding-Iohara algebra is an elliptic analog of the Ding-Iohara algebra introduced by the author. 

\begin{definition}[\textbf{Elliptic Ding-Iohara algebra $\mathcal{U}(q,t,p)$ [14]}] 
Set the structure function $g_{p}(x)$ as
\begin{align*}
g_{p}(x):=\frac{\Theta_{p}(qx)\Theta_{p}(t^{-1}x)\Theta_{p}(q^{-1}tx)}{\Theta_{p}(q^{-1}x)\Theta_{p}(tx)\Theta_{p}(qt^{-1}x)}.
\end{align*}
Then $g_{p}(x^{-1})=g_{p}(x)^{-1}$. Let $x^{\pm}(p;z):=\sum_{n\in{\mathbf{Z}}}x^{\pm}_{n}(p)z^{-n}$, $\psi^{\pm}(p;z):=\sum_{n\in{\mathbf{Z}}}\psi^{\pm}_{n}(p)z^{-n}$ be currents and $\gamma$ be central, invertible element satisfying the following relations :
\begin{align}
&\hskip 1cm [\psi^{\pm}(p;z), \psi^{\pm}(p;w)]=0, \quad \psi^{+}(p;z)\psi^{-}(p;w)=\frac{g_{p}(\gamma z/w)}{g_{p}(\gamma^{-1}z/w)}\psi^{-}(p;w)\psi^{+}(p;z),\notag\\
&\hskip 3cm \psi^{\pm}(p;z)x^{+}(p;w)=g_{p}\Big(\gamma^{\pm\frac{1}{2}}\frac{z}{w}\Big)x^{+}(p;w)\psi^{\pm}(p;z),\notag\\
&\hskip 3cm \psi^{\pm}(p;z)x^{-}(p;w)=g_{p}\Big(\gamma^{\mp\frac{1}{2}}\frac{z}{w}\Big)^{-1}x^{-}(p;w)\psi^{\pm}(p;z),\notag\\
&\hskip 3cm x^{\pm}(p;z)x^{\pm}(p;w)=g_{p}\Big(\frac{z}{w}\Big)^{\pm 1}x^{\pm}(p;w)x^{\pm}(p;z),\notag\\
&[x^{+}(p;z),x^{-}(p;w)]
=\frac{\Theta_{p}(q)\Theta_{p}(t^{-1})}{(p;p)_{\infty}^{3}\Theta_{p}(qt^{-1})}\bigg\{\delta\Big(\gamma\frac{w}{z}\Big)\psi^{+}(p;\gamma^{1/2}w){-}\delta\Big(\gamma^{-1}\frac{w}{z}\Big)\psi^{-}(p;\gamma^{-1/2}w)\bigg\}.
\end{align}
Here we define the delta function as $\delta(x):=\sum_{n\in\mathbf{Z}}x^{n}$. We define the elliptic Ding-Iohara algebra $\mathcal{U}(q,t,p)$ to be an associative $\mathbf{C}$-algebra generated by $\{x^{\pm}_{n}(p)\}_{n\in{\mathbf{Z}}}$, $\{\psi^{\pm}_{n}(p)\}_{n\in{\mathbf{Z}}}$ and $\gamma$. 
\end{definition}

\begin{theorem}[\textbf{Free field realization of the elliptic Ding-Iohara algebra $\mathcal{U}(q,t,p)$ [14]}] 
Set $\gamma:=(qt^{-1})^{-1/2}$ and operators $\varphi^{\pm}(p;z) : \mathcal{F} \to \mathcal{F}\otimes\mathbf{C}[[z,z^{-1}]]$ as 
\begin{align*}
\varphi^{+}(p;z):=\bm{:}\eta(p;\gamma^{1/2}z)\xi(p;\gamma^{-1/2}z)\bm{:}, \quad \varphi^{-}(p;z):=\bm{:}\eta(p;\gamma^{-1/2}z)\xi(p;\gamma^{1/2}z)\bm{:}.
\end{align*}
Then the map
\begin{align*}
x^{+}(p;z) \mapsto \eta(p;z), \quad x^{-}(p;z) \mapsto \xi(p;z), \quad \psi^{\pm}(p;z) \mapsto \varphi^{\pm}(p;z)
\end{align*}
gives a representation of the elliptic Ding-Iohara algebra $\mathcal{U}(q,t,p)$.
\end{theorem}

\subsection{Elliptic Feigin-Odesskii algebra $\mathcal{A}(p)$}
The elliptic Feigin-Odesskii algebra is defined in the similar way of the trigonometric case except the emergence of elliptic functions [8].

\begin{definition}[\textbf{Elliptic Feigin-Odesskii algebra $\mathcal{A}(p)$}] 
Define a $n$-variable function $\varepsilon_{n}(q,p;x) \,(n\in\mathbf{Z}_{>0})$ as follows.
\begin{align}
\varepsilon_{n}(q,p;x):=\prod_{1\leq{a}<b\leq{n}}\frac{\Theta_{p}(qx_{a}/x_{b})\Theta_{p}(q^{-1}x_{a}/x_{b})}{\Theta_{p}(x_{a}/x_{b})^{2}}.
\end{align}
Set a function $\omega_{p}(x,y)$ as
\begin{align}
\omega_{p}(x,y):=\frac{\Theta_{p}(q^{-1}y/x)\Theta_{p}(ty/x)\Theta_{p}(qt^{-1}y/x)}{\Theta_{p}(y/x)^{3}}.
\end{align}
Define the star product $\ast$ as
\begin{align}
(f\ast g)(x_{1},\cdots,x_{m+n}):=\text{Sym}\bigg[f(x_{1},\cdots,x_{m})g(x_{m+1},\cdots,x_{m+n})\prod_{\genfrac{}{}{0pt}{1}{1\leq{\alpha}\leq{m}}{m+1\leq{\beta}\leq{m+n}}}\omega_{p}(x_{\alpha},x_{\beta})\bigg].
\end{align}
For a partition $\lambda$, we set $\varepsilon_{\lambda}(q,p;x)$ as
\begin{align}
\varepsilon_{\lambda}(q,p;x):=\varepsilon_{\lambda_{1}}(q,p;x)\ast\cdots\ast\varepsilon_{\lambda_{\ell(\lambda)}}(q,p;x).
\end{align}
Set $\mathcal{A}_{0}(p):=\mathbf{C}$, $\mathcal{A}_{n}(p):=\text{span}\{\varepsilon_{\lambda}(q,p;x) : |\lambda|=n\}$ $(n\geq{1})$. We define the elliptic Feigin-Odesskii algebra as $\mathcal{A}(p):=\bigoplus_{n\geq{0}}\mathcal{A}_{n}(p)$ whose algebraic structure is given by the star product $\ast$.
\end{definition}

Similar to the trigonometric case, the following is shown [8]. 

\begin{prop}[] 
The elliptic Feigin-Odesskii algebra $(\mathcal{A}(p),\ast)$ is unital, associative, and commutative algebra.
\end{prop}

\subsection{Commutative families $\mathcal{M}(p)$, $\mathcal{M}^{\prime}(p)$}
For the operators $E(p;z)$, $F(p;z)$ which are used in the theorem 3.2, we have the following proposition [14].

\begin{prop}
(1) Operators $E(p;z)$, $F(p;z)$ satisfy the relation as
\begin{align}
E(p;z)E(p;w)&=g_{p}\Big(\frac{z}{w}\Big)E(p;w)E(p;z), \\
F(p;z)F(p;w)&=g_{p}\Big(\frac{z}{w}\Big)^{-1}F(p;w)F(p;z).
\end{align}
Due to the relations operator-valued functions as
\begin{align*}
\prod_{1\leq{i<j}\leq{N}}\omega_{p}(x_{i},x_{j})^{-1}E(p;x_{1})\cdots E(p;x_{N}), \quad \prod_{1\leq{i<j}\leq{N}}\omega^{\prime}_{p}(x_{i},x_{j})^{-1}F(p;x_{1})\cdots F(p;x_{N})
\end{align*}
are symmetric in $x_{1},\cdots,x_{N}$.

(2) Commutator of $E(p;z)$ and $F(p;z)$ takes the following form.
\begin{align}
[E(p;z),F(p;w)]=\frac{\Theta_{p}(q)\Theta_{p}(t^{-1})}{(p;p)_{\infty}^{3}\Theta_{p}(qt^{-1})}\delta\Big(\gamma\frac{w}{z}\Big)\{\varphi^{+}(p;\gamma^{1/2}w)-\varphi^{+}(p;\gamma^{1/2}p^{-1}w)\}.
\end{align}
\end{prop}

From the relation (3.26) we have $[[E(p;z)]_{1},[F(p;w)]_{1}]=0$. This corresponds to the commutativity of the elliptic Macdonald operators $[H_{N}(q,t,p),H_{N}(q^{-1},t^{-1},p)]=0$. 

\begin{definition}[\textbf{Map $\mathcal{O}_{p}$}] 
We define a linear map $\mathcal{O}_{p} : \mathcal{A}(p) \to \text{End}(\mathcal{F}_{\alpha})$ $(\alpha\in\mathbf{C})$ as follows.
\begin{align}
\mathcal{O}_{p}(f):=\bigg[f(z_{1},\cdots,z_{n})\prod_{1\leq{i<j}\leq{n}}\omega_{p}(z_{i},z_{j})^{-1}E(p;z_{1})\cdots E(p;z_{n})\bigg]_{1} \quad(f\in\mathcal{A}_{n}(p)).
\end{align}
Here $[f(z_{1},\cdots,z_{n})]_{1}$ denotes the constant term of $f(z_{1},\cdots,z_{n})$ in $z_{1},\cdots,z_{n}$. We extend the map linearly.
\end{definition}

In the similar way of the trigonometric case, we can check the following.

\begin{prop}[] 
The map $\mathcal{O}_{p}$ and the star product $\ast$ are compatible : for $f,g\in\mathcal{A}(p)$, we have $\mathcal{O}_{p}(f\ast g)=\mathcal{O}_{p}(f)\mathcal{O}_{p}(g)$.
\end{prop}

\begin{theorem}[\textbf{Commutative family $\mathcal{M}(p)$}]
(1) Set $\mathcal{M}(p):=\mathcal{O}_{p}(\mathcal{A}(p))$. The space is a commutative algebra of boson operators.

(2) The space $\mathcal{M}(p)|_{\mathbf{C}\phi_{N}(p;x)|N \rangle}$ is a set of commuting elliptic $q$-difference operators which contains the elliptic Macdonald operator $H_{N}(q,t,p)$ (commutative family of the elliptic Macdonald operator $H_{N}(q,t,p)$).
\end{theorem}

A commutative family of the elliptic Macdonald operator $H_{N}(q^{-1},t^{-1},p)$ is also constructed as follows.

\begin{definition}[\textbf{Map $\mathcal{O}_{p}^{\prime}$}] 
Set a function $\omega_{p}^{\prime}(x,y)$ as
\begin{align}
\omega_{p}^{\prime}(x,y):=\frac{\Theta_{p}(qy/x)\Theta_{p}(t^{-1}y/x)\Theta_{p}(q^{-1}ty/x)}{\Theta_{p}(y/x)^{3}}.
\end{align}
We define a linear map $\mathcal{O}_{p}^{\prime} : \mathcal{A}(p) \to \text{End}(\mathcal{F}_{\alpha})$ $(\alpha\in\mathbf{C})$ as follows.
\begin{align}
\mathcal{O}_{p}^{\prime}(f):=\bigg[f(z_{1},\cdots,z_{n})\prod_{1\leq{i<j}\leq{n}}\omega_{p}^{\prime}(z_{i},z_{j})^{-1}F(p;z_{1})\cdots F(p;z_{n})\bigg]_{1} \quad(f\in\mathcal{A}_{n}(p)).
\end{align}
We extend the map linearly.
\end{definition}

In the same way of the trigonometric case, we have the following lemma.

\begin{lemma}
Set another star product $\ast^{\prime}$ as
\begin{align}
(f\ast^{\prime} g)(x_{1},\cdots,x_{m+n}):=\text{\rm Sym}\bigg[f(x_{1},\cdots,x_{m})g(x_{m+1},\cdots,x_{m+n})\prod_{\genfrac{}{}{0pt}{1}{1\leq{\alpha}\leq{m}}{m+1\leq{\beta}\leq{m+n}}}\omega^{\prime}_{p}(x_{\alpha},x_{\beta})\bigg].
\end{align}
In the elliptic Feigin-Odesskii algebra $\mathcal{A}(p)$, we have $\ast^{\prime}=\ast$.
\end{lemma}

We also have the following theorem [15].

\begin{theorem}[\textbf{Commutative family $\mathcal{M}^{\prime}(p)$}]
(1) Set $\mathcal{M}^{\prime}(p):=\mathcal{O}_{p}^{\prime}(\mathcal{A}(p))$. The space is a commutative algebra of boson operators. 

(2) The space $\mathcal{M}^{\prime}(p)|_{\mathbf{C}\phi_{N}(p;x)|N \rangle}$ is a set of commuting elliptic $q$-difference operators which contains the elliptic Macdonald operator $H_{N}(q^{-1},t^{-1},p)$ (commutative family of the elliptic Macdonald operator $H_{N}(q^{-1},t^{-1},p)$).
\end{theorem}

Similar to the proposition 2.17, we can show that the commutative families $\mathcal{M}(p)$, $\mathcal{M}^{\prime}(p)$ commute with each other [15].

\begin{theorem}
For the commutative families $\mathcal{M}(p)$, $\mathcal{M}^{\prime}(p)$, we have the commutativity $[\mathcal{M}(p), \mathcal{M}^{\prime}(p)]=0$.
\end{theorem}

The theorem 3.18 is the elliptic analog of the proposition 2.17. But we can't prove the theorem 3.18 in the similar way of the proof of the proposition 2.17, because we don't have an elliptic analog of the Macdonald symmetric functions. Hence we will show the theorem 3.18 in a direct way. For the proof we prepare the following lemma.

\begin{lemma}
Assume that a $r$-variable function $A(x_{1},\cdots,x_{r})$ and a $s$-variable function $B(x_{1},\cdots,x_{s})$ have a period $p$, \textit{i.e.} $T_{p,x_{i}}A(x_{1},\cdots,x_{r})=A(x_{1},\cdots,x_{r}) \, (1\leq{i}\leq{r})$, $T_{p,x_{i}}B(x_{1},\cdots,x_{s})=B(x_{1},\cdots,x_{s}) \, (1\leq{i}\leq{s})$. Then we have 
\begin{align}
[[A(z_{1},\cdots,z_{r})E(p;z_{1})\cdots E(p;z_{r})]_{1},[B(w_{1},\cdots,w_{s})F(p;w_{1})\cdots F(p;w_{s})]_{1}]=0.
\end{align}
\end{lemma}

\begin{proof}[\textit{Proof}] 
First we have
\begin{align}
[A_{1}\cdots A_{r},B_{1}\cdots B_{s}]=\sum_{i=1}^{r}\sum_{j=1}^{s}A_{1}\cdots A_{i-1}B_{1}\cdots B_{j-1}[A_{i},B_{j}]B_{j+1}\cdots B_{s}A_{i+1}\cdots A_{r}.
\end{align}
Set $c(q,t,p):=\Theta_{p}(q)\Theta_{p}(t^{-1})/(p;p)_{\infty}^{3}\Theta_{p}(qt^{-1})$ and we denote the $p$-difference of $f(z)$ by $\Delta_{p}f(z):=f(pz)-f(z)$. By the equation (3.32) and (3.26) in the proposition 3.11, we have the following.
\begin{align}
&\quad [A(z_{1},\cdots,z_{r})E(p;z_{1})\cdots E(p;z_{r}),B(w_{1},\cdots,w_{s})F(p;w_{1})\cdots F(p;w_{s})] \notag\\
&=\sum_{i=1}^{r}\sum_{j=1}^{s}A(z_{1},\cdots,z_{r})B(w_{1},\cdots,w_{s})E(p;z_{1})\cdots E(p;z_{i-1})F(p;w_{1})\cdots F(p;w_{j-1}) \notag\\
&\hskip 2cm \times [E(p;z_{i}),F(p;w_{j})]F(p;w_{j+1})\cdots F(p;w_{s})E(p;z_{i+1})\cdots E(p;z_{r}) \notag\\
&=c(q,t,p)\sum_{i=1}^{r}\sum_{j=1}^{s}E(p;z_{1})\cdots E(p;z_{i-1})F(p;w_{1})\cdots F(p;w_{j-1}) \notag\\
&\quad\times A(z_{1},\cdots,\overbrace{\gamma w_{j}}^{\text{$i$-th}},\cdots,z_{r})B(w_{1},\cdots,\overbrace{w_{j}}^{\text{$j$-th}},\cdots,w_{s})\delta\Big(\gamma\frac{w_{j}}{z_{i}}\Big)\Delta_{p}\varphi^{+}(p;\gamma^{1/2}p^{-1}w_{j}) \notag\\
&\quad\times F(p;w_{j+1})\cdots F(p;w_{s})E(p;z_{i+1})\cdots E(p;z_{r}).
\end{align}
By picking up the constant term of $z_{i}$, $w_{j}$ dependence part of (3.33), we have
\begin{align}
&\quad \bigg[A(z_{1},\cdots,\overbrace{\gamma w_{j}}^{\text{$i$-th}},\cdots,z_{r})B(w_{1},\cdots,\overbrace{w_{j}}^{\text{$j$-th}},\cdots,w_{s})\delta\Big(\gamma\frac{w_{j}}{z_{i}}\Big)\Delta_{p}\varphi^{+}(p;\gamma^{1/2}p^{-1}w_{j})\bigg]_{z_{i},w_{j},1} \notag\\
&=\bigg[A(z_{1},\cdots,\overbrace{\gamma w_{j}}^{\text{$i$-th}},\cdots,z_{r})B(w_{1},\cdots,\overbrace{w_{j}}^{\text{$j$-th}},\cdots,w_{s})\Delta_{p}\varphi^{+}(p;\gamma^{1/2}p^{-1}w_{j})\bigg]_{w_{j},1} \notag\\
&=\bigg[A(z_{1},\cdots,\overbrace{\gamma w_{j}}^{\text{$i$-th}},\cdots,z_{r})B(w_{1},\cdots,\overbrace{w_{j}}^{\text{$j$-th}},\cdots,w_{s})\varphi^{+}(p;\gamma^{1/2}w_{j})\bigg]_{w_{j},1} \notag\\
&\hskip 1cm -\bigg[A(z_{1},\cdots,\overbrace{\gamma w_{j}}^{\text{$i$-th}},\cdots,z_{r})B(w_{1},\cdots,\overbrace{w_{j}}^{\text{$j$-th}},\cdots,w_{s})\varphi^{+}(p;\gamma^{1/2}p^{-1}w_{j})\bigg]_{w_{j},1}. 
\end{align}
We recall $[f(z)]_{1}=[f(az)]_{1}$ $(a\in\mathbf{C})$ and the functions $A(z_{1},\cdots,z_{r})$ and $B(w_{1},\cdots,w_{s})$ have a period $p$. Hence we have
\begin{align*}
\bigg[A(z_{1},\cdots,\overbrace{\gamma w_{j}}^{\text{$i$-th}},\cdots,z_{r})B(w_{1},\cdots,\overbrace{w_{j}}^{\text{$j$-th}},\cdots,w_{s})\delta\Big(\gamma\frac{w_{j}}{z_{i}}\Big)\Delta_{p}\varphi^{+}(p;\gamma^{1/2}p^{-1}w_{j})\bigg]_{z_{i},w_{j},1}=0.
\end{align*}
The equation holds for any $i,j$, hence we have the lemma 3.19. \quad $\Box$
\end{proof}

\begin{proof}[\textit{Proof of the theorem 3.18.}]
For the proof what we have to show is 
\begin{align*}
[\mathcal{O}_{p}(\varepsilon_{r}(q,p;z)),\mathcal{O}^{\prime}_{p}(\varepsilon_{s}(q,p;w))]=0 \quad (r,s\in\mathbf{Z}_{>0}).
\end{align*}
By the definition of $\mathcal{O}_{p}$, $\mathcal{O}^{\prime}_{p}$, operators $\mathcal{O}_{p}(\varepsilon_{r}(q,p;z))$, $\mathcal{O}^{\prime}_{p}(\varepsilon_{s}(q,p;w))$ are the constant terms of the following operators.
\begin{align}
&\varepsilon_{r}(q,p;z)\prod_{1\leq{i<j}\leq{r}}\omega_{p}(z_{i},z_{j})^{-1}E(p;z_{1})\cdots E(p;z_{r}), \\
&\varepsilon_{s}(q,p;w)\prod_{1\leq{i<j}\leq{s}}\omega^{\prime}_{p}(w_{i},w_{j})^{-1}F(p;w_{1})\cdots F(p;w_{s}).
\end{align}
Then their functional parts take the following forms.
\begin{align}
&\quad\text{(Functional part of (3.35))} \notag\\
&=\varepsilon_{r}(q,p;z)\prod_{1\leq{i<j}\leq{r}}\omega_{p}(z_{i},z_{j})^{-1}
=\prod_{1\leq{i<j}\leq{r}}\frac{\Theta_{p}(z_{i}/z_{j})\Theta_{p}(q^{-1}z_{i}/z_{j})}{\Theta_{p}(t^{-1}z_{i}/z_{j})\Theta_{p}(q^{-1}tz_{i}/z_{j})}, \\
&\quad\text{(Functional part of (3.36))} \notag\\
&=\varepsilon_{s}(q,p;w)\prod_{1\leq{i<j}\leq{s}}\omega^{\prime}_{p}(w_{i},w_{j})^{-1}
=\prod_{1\leq{i<j}\leq{s}}\frac{\Theta_{p}(w_{i}/w_{j})\Theta_{p}(qw_{i}/w_{j})}{\Theta_{p}(tw_{i}/w_{j})\Theta_{p}(qt^{-1}w_{i}/w_{j})}.
\end{align}
We can check (3.37), (3.38) have a period $p$. By the lemma 3.19, we have the theorem 3.18. \quad $\Box$
\end{proof}

By the theorem 3.18, commutative families $\mathcal{M}(p)|_{\mathbf{C}\phi_{N}(p;x)|N \rangle}$, $\mathcal{M}^{\prime}(p)|_{\mathbf{C}\phi_{N}(p;x)|N \rangle}$ also commute with each other : $[\mathcal{M}(p)|_{\mathbf{C}\phi_{N}(p;x)|N \rangle}, \mathcal{M}^{\prime}(p)|_{\mathbf{C}\phi_{N}(p;x)|N \rangle}]=0$.

\medskip
\textbf{Acknowledgement.}
The author would like to thank Koji Hasegawa and Gen Kuroki for helpful discussions and comments.

\end{document}